\documentclass[11pt]{amsart}

\usepackage{amsmath}
\usepackage[bookmarks,colorlinks=false]{hyperref}
\usepackage{cite}

\renewcommand{\mod}{\; \mathrm{mod} \;}

\newcommand{\set}[1]{\left\{ #1 \right\}}
\newcommand{\abs}[1]{\left| #1 \right|}
\newcommand{\ol}[1]{\overline{#1}}

\DeclareMathOperator{\Dist}{Dist}
\DeclareMathOperator{\End}{End}
\DeclareMathOperator{\gr}{gr}
\DeclareMathOperator{\opH}{H}

\DeclareMathOperator{\ind}{ind}
\DeclareMathOperator{\Lie}{Lie}
\DeclareMathOperator{\soc}{soc}
\DeclareMathOperator{\St}{St}

\newcommand{\C}{\mathbb{C}}
\newcommand{\F}{\mathbb{F}}
\newcommand{\G}{\mathbb{G}}
\newcommand{\N}{\mathbb{N}}

\newcommand{\Z}{\mathbb{Z}}

\newcommand{\g}{\mathfrak{g}}
\newcommand{\fu}{\mathfrak{u}}

\newcommand{\gc}{\g_\C}

\newcommand{\Ugc}{U(\gc)}
\newcommand{\Uzg}{U_\Z(\g)}

\newcommand{\Fp}{\F_p}
\newcommand{\Fq}{\F_q}

\newcommand{\gfp}{\g_{\Fp}}
\newcommand{\gfq}{\g_{\Fq}}
\newcommand{\fufp}{\fu_{\Fp}}
\newcommand{\fufq}{\fu_{\Fq}}
\newcommand{\ofp}{\otimes_{\Fp}}

\newcommand{\Gfp}{G(\Fp)}
\newcommand{\Gfq}{G(\Fq)}

\newcommand{\Bfq}{B(\Fq)}
\newcommand{\Ufp}{U(\Fp)}
\newcommand{\Ufq}{U(\Fq)}

\numberwithin{equation}{section}

\newtheorem{theorem}{Theorem}[section]

\newtheorem{corollary}[theorem]{Corollary}
\newtheorem{lemma}[theorem]{Lemma}

\newtheorem*{theorem*}{Theorem}
\newtheorem*{lemma*}{Lemma}

\theoremstyle{definition}
\newtheorem{example}[theorem]{Example}

\title[On projective modules]{On projective modules for Frobenius kernels and finite Chevalley groups}

\thanks{The author was supported in part by NSF VIGRE grant DMS-0738586.}

\author{Christopher M.\ Drupieski}

\date{\today}

\address{
Department of Mathematics \\
DePaul University \\
Chicago, IL 60614}
\email{cdrupies@depaul.edu}


\begin{document}

\begin{abstract}
Let $G$ be a simply-connected semisimple algebraic group scheme over an algebraically closed field of characteristic $p > 0$. Let $r \geq 1$ and set $q = p^r$. We show that if a rational $G$-module $M$ is projective over the $r$-th Frobenius kernel $G_r$ of $G$, then it is also projective when considered as a module for the finite subgroup $\Gfq$ of $\Fq$-rational points in $G$. This salvages a theorem of Lin and Nakano (\emph{Bull.\ London Math.\ Soc.} 39 (2007) 1019--1028). We also show that the corresponding statement need not hold when the group $G$ is replaced by the unipotent radical $U$ of a Borel subgroup $B$ of $G$.
\end{abstract}

\maketitle

\section{Introduction} \label{section:introduction}

\subsection{} \label{subsection:introduction}

Let $G$ be a simply-connected semisimple algebraic group scheme over an algebraically closed field $k$ of characteristic $p > 0$. Assume that $G$ is defined and split over $\Fp$, and let $F: G \rightarrow G$ be the standard Frobenius morphism defining the $\Fp$-structure on $G$. Given an integer $r \geq 1$, let $G_r = \ker(F^r)$ be the $r$-th infinitesimal Frobenius kernel of $G$, and let $\Gfq = G^{F^r}$ be the finite subgroup of $\Fq$-rational points in $G$, consisting of the fixed points in $G$ under the $r$-th iterate of $F$. Here $q = p^r$. Let $\g = \Lie(G)$ be the Lie algebra of $G$, and let $u(\g)$ be the restricted enveloping algebra of $\g$.

In the proceedings of the 1986 Arcata Conference on Representations of Finite Groups, Brian Parshall asked whether a finite-dimensional rational $G$-module that is projective for $G_1$ (equivalently, for $u(\g)$) is always projective for $\Gfp$ \cite[5.3]{Parshall:1987}. Lin and Nakano provided an affirmative answer to this question in 1999 by showing that if $M$ is a rational $G$-module, then the complexity $c_{\Gfp}(M)$ of $M$ as a $k\Gfp$-module is at most one-half the complexity $c_{G_1}(M)$ of $M$ as a $G_1$-module \cite[Theorem 3.4]{Lin:1999}. Since a module is projective if and only if its complexity is zero, this observation answered Parshall's question. The complexities $c_{\Gfp}(M)$ and $c_{G_1}(M)$ can also be interpreted as the dimensions of the associated cohomological support varieties $\abs{\Gfp}_M$ and $\abs{G_1}_M$. Thus, the Lin--Nakano approach possesses a certain geometric flavor, and subsequent work by Carlson, Lin, and Nakano \cite{Carlson:2008} and by Friedlander \cite{Friedlander:2010,Friedlander:2011} has sought to better understand the relationship between the varieties $\abs{\Gfp}_M$ and $\abs{G_1}_M$.

In this note we provide an affirmative answer to Parshall's question for all $r \geq 1$. Specifically, given a finite-dimensional rational $G$-module $M$, we show that if $M$ is projective for the $r$-th Frobenius kernel $G_r$ of $G$, then $M$ is projective as a $k\Gfq$-module where $q = p^r$.\footnote{This statement is referred to in \cite{Lin:2007} as the \emph{Generalized Parshall Conjecture}.} This generalization was previously claimed by Lin and Nakano in 2007 \cite{Lin:2007}, though their argument was incomplete because of an error in the proof of their key proposition; for a more detailed explanation see Section \ref{subsection:exampleU}. The argument we present here is entirely non-geometric in nature, that is, it does not require the use or discussion of support varieties or complexity, but relies instead only on the algebra of distributions on $G$, so is interesting even for the previously-established case when $r=1$. For $r=1$ our argument also eliminates certain assumptions on the prime $p$ that were necessary for the methods of \cite{Lin:1999}. After proving the main theorem, we provide in Section \ref{subsection:exampleU} an example to show that the corresponding statement need not hold when $G$ is replaced by the unipotent radical $U$ of a Borel subgroup $B$ of $G$. Finally, in Section \ref{section:Friedlander} we discuss some recent results of Friedlander \cite{Friedlander:2010} that are related to the projectivity of modules over $\Gfq$.

\subsection{Notation}

Let $G$, $F$, $G_r$, $\Gfq$, and $\g$ be as defined in Section \ref{subsection:introduction}. Let $T \subset G$ be a maximal torus defined and split over $\Fp$, and let $\Phi$ be the set of roots of $T$ in $G$. Let $B \subset G$ be a Borel subgroup containing $T$ and corresponding to the set $\Phi^-$ of negative roots in $\Phi$, and let $B^+ \subset G$ be the opposite Borel subgroup corresponding to the set $\Phi^+$ of positive roots in $\Phi$. Let $X(T)$ be the integral weight lattice obtained from $T$. Then $X(T)$ is partially ordered by $\lambda \geq \mu$ if and only if $\lambda - \mu$ is a sum of positive roots. Let $U \subset B$ be the unipotent radical of $B$. Set $\N = \set{0,1,2,\ldots}$. Additional notation will be introduced as needed.

\section{The Generalized Parshall Conjecture}

\subsection{The algebra of distributions}

We begin by recalling certain basic facts concerning the algebra of distributions on $G$; for further details, see \cite[II.1.12, II.1.19, II.3.3]{Jantzen:2003}. Let $\gc$ be the complex semisimple Lie algebra of the same Lie type as $G$, with Chevalley basis $\{X_\alpha,H_i: \alpha \in \Phi, i \in [1,n] \}$. Let $\Ugc$ be the universal enveloping algebra of $\gc$, and let $\Uzg \subset \Ugc$ be the Kostant $\Z$-form of $\Ugc$. Since $G$ is semisimple and simply-connected, the algebra $\Dist(G)$ of distributions on $G$ with support at the identity, also known as the hyperalgebra of $G$, may be realized as $\Uzg \otimes_\Z k$, the $k$-algebra obtained via scalar extension from $\Uzg$. Thus, $\Dist(G)$ admits a $k$-basis consisting of all monomials
\[
\prod_{\alpha \in \Phi^-} X_{\alpha,n(\alpha)} \prod_{i=1}^n H_{i,m(i)} \prod_{\alpha \in \Phi^+} X_{\alpha,n'(\alpha)},
\]
where $n(\alpha),m(i),n'(\alpha) \in \N$, $X_{\alpha,n} = X_\alpha^n/(n!)$, $H_{i,m} = \binom{H_i}{m}$, and the products are taken with respect to any fixed ordering of the roots. Similarly, $\Dist(U)$ admits a $k$-basis consisting of all monomials $\prod_{\alpha \in \Phi^-} X_{\alpha,n(\alpha)}$ with $n(\alpha) \in \N$. If the integers $n(\alpha)$ are restricted to lie in the range $0 \leq n(\alpha) < p^r$, then one obtains a $k$-basis for the algebra $\Dist(U_r)$. Given $\alpha \in \Phi$, the vectors $X_{\alpha,n}$ with $n \in \N$ form a $k$-basis for the algebra $\Dist(U_\alpha)$ of distributions on the one-dimensional root subgroup $U_\alpha$, and the vectors $X_{\alpha,n}$ with $0 \leq n < p^r$ form a $k$-basis for $\Dist(U_{\alpha,r})$, the algebra of distributions on the $r$-th Frobenius kernel $U_{\alpha,r}$ of $U_\alpha$.

Each rational $U_\alpha$-module $M$ naturally admits the structure of a locally finite $\Dist(U_\alpha)$-module. Moreover, the action of the $X_{\alpha,n}$ on $M$ determines the action of $U_\alpha$ on $M$. Indeed, let $x_\alpha: \G_a \rightarrow U_\alpha$ be a fixed isomorphism between the additive group $\G_a$ and the root subgroup $U_\alpha$. Then for $a \in \G_a$, the action of $x_\alpha(a) \in U_\alpha$ on $m \in M$ is related to the action of the $X_{\alpha,n}$ on $m$ by the equation
\[
x_\alpha(a).m = \sum_{n \geq 0} a^n (X_{\alpha,n}.m).
\]

\subsection{An equality of endomorphism spaces} \label{subsection:equivalentspan}

For each $1 \leq i < q$, define the formal infinite sum
\[
y_{\alpha,i} = \sum_{n \geq 0} X_{\alpha,i+n(q-1)},
\]
and set $y_{\alpha,0} = 1$. Then the $y_{\alpha,i}$ are well-defined operators on any rational $U_\alpha$-module.

\begin{lemma} \label{lemma:equivalentspan}
Let $M$ be a rational $U_\alpha$-module. Then the span in $\End_k(M)$ of the operators $y_{\alpha,0},y_{\alpha,1},\ldots,y_{\alpha,q-1}$ is the same as the $k$-span of the operators $\set{x_\alpha(a):a \in \Fq}$.
\end{lemma}

\begin{proof}
Since every rational $U_\alpha$-module is a sum of finite-dimensional modules, it suffices to assume that $M$ is finite-dimensional. Then there exists an integer $N \geq q$ such that for all $m \in M$, $X_{\alpha,n}.m = 0$ for all $n \geq N$. Then $x_\alpha(a)$ acts on $M$ via the finite sum $\ol{x_\alpha(a)}:= \sum_{n=0}^{N-1} a^n X_{\alpha,n} \in \Dist(U_\alpha)$. Similarly, $y_{\alpha,i}$ acts on $M$ via the finite sum
\[
\ol{y_{\alpha,i}}:= \sum_{\substack{0 \leq m < N \\ m \equiv i \mod q-1}} X_{\alpha,m},
\]
and the $\ol{y_{\alpha,i}}$ are linearly independent elements of $\Dist(U_\alpha)$.

Let $a \in \Fq$. Since $a^q = a$, we have $\ol{x_\alpha(a)} = \sum_{i=0}^{q-1} a^i \ol{y_{\alpha,i}}$, where by convention we set $0^0 = 1$. Writing $\Fq = \set{a_0,a_1,\ldots,a_{q-1}}$, the matrix for the linear transformation that sends $\ol{y_{\alpha,i}} \mapsto \ol{x_\alpha(a_i)}$ is an invertible Vandermonde matrix; cf.\ \cite[\S 3.3]{Lin:2007}. It follows that $\set{\ol{y_{\alpha,i}}: 0 \leq i < q}$ and $\{\ol{x_\alpha(a)}:a \in \Fq \}$ are each linearly independent sets spanning the same subspace of $\Dist(U_\alpha)$, and consequently that their images span the same subspace of $\End_k(M)$.
\end{proof}

\subsection{The Generalized Parshall Conjecture for Borel subgroups}

We now establish the Generalized Parshall Conjecture for the Borel subgroup $B$ of $G$.

\begin{theorem} \label{theorem:GPCforB}
Let $M$ be a finite-dimensional rational $B$-module. Suppose $M$ is projective as a $B_r$-module. Then $M$ is projective as a $k\Ufq$-module, and hence also as a $k\Bfq$-module.
\end{theorem}

\begin{proof}
Suppose $M$ is projective as a $B_r$-module. Then $M$ is a projective $B_rT$-module \cite[II.9.4]{Jantzen:2003}, and by the explicit description of the projective indecomposable $B_rT$-modules \cite[II.9.5]{Jantzen:2003}, there exists a $\Dist(U_r)$-basis $\set{m_1,\ldots,m_s}$ for $M$ consisting of weight vectors for $T$. Because $\Ufq$ is a Sylow $p$-subgroup of $\Bfq$, a $k\Bfq$-module is projective if and only if it is projective (hence, free) as a $k\Ufq$-module. Then to prove the theorem it suffices to show that $\set{m_1,\ldots,m_s}$ is also a $k\Ufq$-basis for $M$. Since $\dim_k \Dist(U_r) = \dim_k k \Ufq$, to show that $\set{m_1,\ldots,m_s}$ is a $k\Ufq$-basis for $M$, it suffices to show that the set $\set{m_1,\ldots,m_s}$ generates $M$ as a $k\Ufq$-module. Using the partial order on $X(T)$, and the operators defined in Section \ref{subsection:equivalentspan}, we argue by induction on the weight ordering to show that the $k\Ufq$-span $M'$ of the set $\set{m_1,\ldots,m_s}$ contains all weight vectors in $M$, hence is equal to $M$.

To begin, fix an enumeration $\Phi^- = \set{\alpha_1,\ldots,\alpha_N}$, and let $\lambda \in X(T)$ be a lowest weight of $T$ in $M$. Since $\set{m_1,\ldots,m_s}$ is a $\Dist(U_r)$-basis for $M$, it follows that the $\lambda$-weight space $M_\lambda$ must be spanned by vectors of the form $X_{\alpha_1,q-1} \cdots X_{\alpha_N,q-1}.m_i$. Since $\lambda$ is a lowest weight vector, we have
\[
X_{\alpha_1,q-1} \cdots X_{\alpha_N,q-1}.m_i = y_{\alpha_1,q-1} \cdots y_{\alpha_N,q-1}.m_i.
\]
Then it follows from Lemma \ref{lemma:equivalentspan} that $M_\lambda \subseteq M'$. Now let $\lambda \in X(T)$ be an arbitrary weight of $T$ in $M$, and set $M_{< \lambda} = \bigoplus_{\mu < \lambda} M_\mu$. By induction, $M_{<\lambda} \subseteq M'$. On the other hand, $M_\lambda$ is spanned by certain vectors of the form $X_{\alpha_1,n_1} \cdots X_{\alpha_N,n_N}.m_i$ with $0 \leq n_i < q$. Given such a vector, the difference
\[
X_{\alpha_1,n_1} \cdots X_{\alpha_N,n_N}.m_i - y_{\alpha_1,n_1} \cdots y_{\alpha_N,n_N}.m_i
\]
is an element of $M_{<\lambda}$, so is a vector in $M'$. But $y_{\alpha_1,n_1} \cdots y_{\alpha_N,n_N}.m_i \in M'$ by Lemma \ref{lemma:equivalentspan}, so we conclude that $X_{\alpha_1,n_1} \cdots X_{\alpha_N,n_N}.m_i \in M'$ as well, and hence that $M_\lambda \subseteq M'$. Since $M$ has only finitely many distinct weight spaces, we conclude that each weight space of $M$ is contained in $M'$, and hence that $M = M'$. Thus, the set $\set{m_1,\ldots,m_s}$ generates $M$ as a $k\Ufq$-module.
\end{proof}

\subsection{Proof of the Generalized Parshall Conjecture}

We now recover the main theorem of \cite{Lin:2007}, and hence also the results contained in \cite[\S\S3--4]{Lin:2007}.

\begin{theorem} \label{theorem:GPC}
Let $G$ be a connected reductive algebraic group over the finite field $\Fq$, and let $M$ be a finite-dimensional rational $G$-module. If $M$ is projective as a $G_r$-module, then $M$ is projective as a $k\Gfq$-module.
\end{theorem}

\begin{proof}
By \cite[Proposition 1.3]{Lin:2007}, it suffices to assume that $G$ is semisimple and simply-connected. Suppose $M$ is projective as a $G_r$-module. Since $B_r$ is a finite group scheme, the induction functor $\ind_{B_r}^{G_r}(-)$ is exact \cite[I.5.13]{Jantzen:2003}, which implies that $M$ is projective (equivalently, injective) for $B_r$ by \cite[I.3.18 and Remark I.4.12]{Jantzen:2003}. Now $M$ is projective as a $k\Ufq$-module by Theorem \ref{theorem:GPCforB}. Since $\Ufq$ is a Sylow $p$-subgroup of $\Gfq$, this implies that $M$ is projective as a $k\Gfq$-module.
\end{proof}

\subsection{Failure of the Generalized Parshall Conjecture for unipotent subgroups} \label{subsection:exampleU}

The following example shows that Theorem \ref{theorem:GPC} need not hold if $G$ is replaced by the unipotent radical $U$ of a Borel subgroup $B$ of $G$.

\begin{example} \label{example:SL2}
Suppose $G = SL_2$, so that $U \cong \G_a$. Then the polynomial $f(t) = t - t^q$ defines an algebraic group homomorphism $f: U \rightarrow U$ with $\ker(f) = U(\Fq)$. Now take $M = f^*(\St_r)$, that is, the rational $U$-module obtained from the $r$-th Steinberg module $\St_r$ by precomposing the $U$-module structure map $U \rightarrow GL(\St_r)$ with $f$. Then $M$ is trivial as a $\Ufq$-module. Let $V$ denote the underlying vector space of $\St_r$. The $k[U]$-comodule structure maps $\Delta_{\St_r},\Delta_M : V \rightarrow V \otimes k[\mathbb{G}_a] \cong V \otimes k[t]$ for $\St_r$ and $M$ are related as follows: If $v \in V$ and $\Delta_{\St_r}(v) = \sum_{i=0}^\infty v_i \otimes t^i$ with $v_i \in V$ and $v_j = 0$ for all $j \gg 0$, then $\Delta_M(v) = \sum_{i=0}^\infty v_i \otimes f(t^i) = \sum_{i=0}^\infty v_i \otimes (t^i-t^{qi})$. It then follows that $M \cong \St_r$ as a $\Dist(U_r)$-module, and hence that $M$ is projective as a $U_r$-module even though it is trivial for the finite group $\Ufq$. Observe that since $f$ is a non-homogenous polynomial, the action of $U$ on $M = f^*(\St_r)$ cannot lift to a rational action of the Borel subgroup $B$.
\end{example}

Because of Example \ref{example:SL2}, it follows that Proposition 2.1 and Theorem 2.3 in \cite{Lin:2007} are false for $H = U$, and that \cite[Corollary 2.4]{Lin:2007} also does not hold for an arbitrary connected algebraic group defined over $\Fq$. The proof of \cite[Proposition 2.1]{Lin:2007} fails because, in the notation used there, a homomorphism vanishing on $\soc_N Q(L)$ need not be the zero map. Whether or not \cite[Proposition 2.1]{Lin:2007} might hold for $H = G$ or $H = B$ remains an open question. Example \ref{example:SL2} also shows that if $U$ is an arbitrary connected unipotent algebraic group scheme defined over $\Fq$, there may exist rational $U$-module structures on $\ind_1^{U_r}(k)$ that are not projective upon restriction to $\Ufq$. It remains an open question whether for such $U$ there always exists \emph{some} rational $U$-module structure on $\ind_1^{U_r}(k)$ that is projective upon restriction to $\Ufq$; see \cite[Conjecture 2.4]{Lin:2007}.

\section{Projectivity and Weil restriction of restricted Lie algebras} \label{section:Friedlander}

\subsection{Restricted Lie algebras arising from filtrations on the group algebra}

Let $G$ be as defined in Section \ref{section:introduction}, and let $M$ be a finite-dimensional rational $G$-module. In their original approach to proving the $r=1$ version of the Parshall Conjecture, Lin and Nakano obtained the inequality $c_{\Gfp}(M) \leq \frac{1}{2} c_{G_1}(M)$ by first proving that $c_{\Ufp}(M) \leq c_{U_1}(M)$. To obtain the latter inequality, they observed that the group ring $k\Ufp$ is filtered by the powers of its augmentation ideal, and that the associated graded algebra $\gr k\Ufp$ is isomorphic to the restricted enveloping algebra $u(\fu)$ for $\fu = \Lie(U)$. Equivalently, $\gr k\Ufp \cong \Dist(U_1)$. They then deduced the existence of a spectral sequence $E_1^{i,j} = \opH^{i+j}(U_1,M)_{(i)} \Rightarrow \opH^{i+j}(\Ufp,M)$, and from this the inequality $c_{U_1}(M) \leq c_{\Ufp}(M)$ followed.

In \cite{Friedlander:2010}, Friedlander applies techniques involving the Weil restriction functor to extend Lin and Nakano's results to the case $r \geq 1$. In this context, the isomorphism $\gr k\Ufp \cong u(\fu)$ is replaced by $\gr k \Ufq \cong u(\fufq \ofp k)$. Here $\fufq$ is the restricted Lie algebra over $\Fq$ obtained via scalar extension to $\Fq$ from a Chevalley basis for $\fu_\C$ (and $\fu_\C$ is the obvious Lie subalgebra of $\gc$ corresponding to $U$). There exists a similar restricted Lie algebra $\fufp$ with $\fufp \ofp \Fq = \fufq$ and $\fufp \ofp k = \fu$. In the isomorphism $\gr k\Ufq \cong u(\fufq \ofp k)$, the Lie algebra $\fufq$ is considered via Weil restriction as a restricted Lie algebra over $\Fp$ (by forgetting the additional $\Fq$-vector space structure), and then the scalars are extended back to $k$. Replacing $\fu$ by $\g$, one also has the restricted Lie algebras $\gfp$, $\gfq = \gfp \ofp \Fq$, and $\gfq \ofp k$. Since
\begin{equation} \label{eq:splitting}
\Fq \ofp \Fq \cong \Fq \times \cdots \times \Fq \quad \text{($r$ times, $q = p^r$),}
\end{equation}
there exists an isomorphism of restricted Lie algebras $\gfq \ofp k \cong \g^{\oplus r}$.

\subsection{Failure of rational modules to be projective}

Let $M$ be a rational $G$-module. The action of $G$ on $M$ differentiates to an action of $\g$, and then restricts to an action of $\gfq$ considered as a restricted Lie algebra over $\Fp$. This action of $\gfq$ on $M$ can be extended over $\Fp$ to an action of $\gfq \ofp k$ on $M$. Then $\gfq \ofp k$ acts on $M$ via the composition of the multiplication map $\gfq \ofp k \rightarrow \g$ with the given action of $\g$ on $M$. With this convention in hand, Friedlander states the following results:

\begin{theorem}[(cf.\ {\cite[Theorem 4.3]{Friedlander:2010}})]
Let $M$ be a rational $G$-module. Then
\begin{equation} \label{eq:complexitybound}
c_{\Gfq}(M) \leq \frac{1}{2} c_{u(\gfq \ofp k)}(M).
\end{equation}
\end{theorem}

\begin{corollary}[(cf.\ {\cite[Corollary 4.4]{Friedlander:2010}})]\label{corollary:projimpliesproj}
Let $M$ be a rational $G$-module. If $M$ is projective for $u(\gfq \ofp k)$, then $M$ is projective for $k\Gfq$.
\end{corollary}

Identifying $\gfq \ofp k$ with $\g^{\oplus r}$, the induced action of $\g^{\oplus r}$ on $M$ is obtained by composing the projection $\g^{\oplus r} \rightarrow \g$ of $\g^{\oplus r}$ onto its first factor with the ordinary action of $\g$ on $M$; this follows from the fact that the multiplication map $k^{\times r} \cong \Fq \ofp k \rightarrow k$ is a $k$-algebra homomorphism, and hence identifies with the projection of $k^{\times r}$ onto one of its factors, say, the first. We use this realization for the action of $\g^{\oplus r}$ on $M$ to show for $r \geq 2$ that a rational $G$-module is never projective over $u(\gfq \ofp k)$, and hence that Corollary~\ref{corollary:projimpliesproj} holds vacuously. In particular, this implies that the generalization of the original Lin--Nakano technique to $r > 1$ is not an effective method for determining the projectivity of a rational $G$-module over $\Gfq$.

\begin{theorem}
Let $M$ be a rational $G$-module, and suppose $r \geq 2$. Then $M$ is not projective for $u(\gfq \ofp k)$.
\end{theorem}

\begin{proof}
Let $M$ be a rational $G$-module, and identify $\gfq \ofp k$ with $\g^{\oplus r}$. Suppose $r \geq 2$. Denote the $p$-th power map on $\g$, that is, the map defining the structure of a $p$-restricted Lie algebra on $\g$, by $x \mapsto x^{[p]}$. Choose $0 \neq x \in \g$ with $x^{[p]} = 0$, and set $z = (0,x,0,\ldots,0) \in \g^{\oplus r}$. Then $z^{[p]} = 0$. Since $\g^{\oplus r}$ acts on $M$ via the first-factor projection map $\g^{\oplus r} \rightarrow \g$ composed with the given action of $\g$ on $M$, one has $z.M = 0$. Let $u(z) \cong k[t]/(t^p)$ be the cyclic subalgebra of $u(\g^{\oplus r})$ generated by $z$. Then $M$ is trivial as a $u(z)$-module, hence not projective over $u(z)$, since every projective $u(z)$-module is free. This implies by \cite[Corollary 1.4]{Friedlander:1986b} that the support variety $\abs{\g^{\oplus r}}_M$ is nonzero, hence by \cite[Proposition 1.5]{Friedlander:1987} that $M$ is not injective (equivalently, projective) for $u(\g^{\oplus r}) = u(\gfq \ofp k)$.
\end{proof}

\providecommand{\bysame}{\leavevmode\hbox to3em{\hrulefill}\thinspace}

\end{document}